\documentclass[12pt]{article} 
\usepackage{setspace}
\usepackage{amsmath, amssymb, amsthm} 
\usepackage{graphicx} 
\usepackage{float}

\title{Local well-posedness and blow-up for the restricted fourth-order Prandtl equation}
\author{
Ik Hyun Choi
}
\usepackage{indentfirst}

\date{}
\begin{document}
\maketitle
\footnotetext{
Email: \texttt{cih0212@snu.ac.kr}



\medskip
The author is grateful to Professor In-Jee Jeong for many helpful discussions
and for the encouragement throughout this work.
}

\begin{abstract}
We prove local well-posedness and finite-time blow-up for a restricted fourth-order Prandtl equation posed on the half-line with clamped boundary conditions. The equation arises from a two\mbox{-} dimensional fourth-order Prandtl system via an ansatz reduction, and its nonlinearity involves a nonlocal integral term. To close a Duhamel fixed-point argument, we need uniform $L^1$ bounds for the associated half-line biharmonic heat kernel. We establish uniform $L^1$ estimates for the kernel and its derivatives, and we show that the semigroup preserves spatial regularity under appropriate compatibility conditions, using an alternative representation derived by integration by parts. These kernel estimates yield local existence and uniqueness for the restricted model and allow us to construct solutions that blow-up in finite time. 
\end{abstract}
\theoremstyle{definition}
\newtheorem{prop}{Theorem}
\newtheorem{lemma}{Lemma}

\newtheorem{corollary}{Corollary}
\theoremstyle{remark}
\newtheorem{remark}{Remark}

\section{Introduction}
\begingroup
\setlength{\parindent}{15pt} 

The Prandtl equation is a central model in the mathematical theory of boundary layers
and has been extensively studied in connection with well-posedness \cite{GerardVaretDormy2010, SammartinoCaflisch1998a} and finite-time blow-up \cite{EEngquist1997, KukavicaVicolWang2017}. Motivated by higher-order dissipative regularizations, we consider the fourth-order
Prandtl-type system on $\mathbb{R}\times \mathbb{R}_+$,
\begin{equation}\label{eq:P4-2D}
\begin{cases}
u_t + u\,u_x + v\,u_y = -u_{yyyy}, & (t,x,y)\in(0,T)\times\mathbb{R}\times\mathbb{R}_+,\\
u_x + v_y = 0, & (t,x,y)\in(0,T)\times\mathbb{R}\times \mathbb{R}_+.
\end{cases}
\end{equation}
supplemented with the clamped boundary conditions
\begin{equation}\label{eq:P4-BC}
u(t,x,0)=0,\qquad u_y(t,x,0)=0,\qquad v(t,x,0)=0,
\end{equation}
and a suitable decay (or matching) condition as $y\to\infty$.

Following the idea in~\cite{EEngquist1997}, we focus on a reduced dynamics generated by a linear-in-$x$ ansatz and construct solutions of the form
\begin{equation}\label{eq:ansatz}
u(t,x,y)=-x\,a(t,y),
\end{equation}
where $a=a(t,y)$ is independent of $x$.
Then $u_x=-a$ and the incompressibility condition yields
\begin{equation}\label{eq:v-from-a}
v(t,x,y)=-\int_0^y u_x(t,x,\eta)\,d\eta=\int_0^y a(t,\eta)\,d\eta .
\end{equation}
Substituting \eqref{eq:ansatz}--\eqref{eq:v-from-a} into \eqref{eq:P4-2D}, we obtain the
one-dimensional nonlocal equation
\begin{equation}\label{eq:a-eq}
a_t = -a_{yyyy} + a^2 - a_y\int_0^y a(t,\eta)\,d\eta,
\qquad (t,y)\in(0,T)\times \mathbb{R}_{+},
\end{equation}
with boundary conditions and decay as $y\to\infty$ inherited from \eqref{eq:P4-BC},
\begin{equation}\label{eq:a-BC}
a(t,0)=0,\qquad a_y(t,0)=0.
\end{equation}

Therefore, local well-posedness and finite-time blow-up for \eqref{eq:a-eq} immediately yield the existence of finite-time blow-up solutions to the two-dimensional system
\eqref{eq:P4-2D}--\eqref{eq:P4-BC} of the form \eqref{eq:ansatz}. Our goal is to establish the existence theory for the restricted fourth-order Prandtl equation \eqref{eq:a-eq}.

Several initial-boundary value problems for fourth-order evolution equations on the half-line have been studied via the unified transform (Fokas) method; see, e.g., \cite{Chatziafratis2025HigherOrder,OzsariYolcu2019CPAA}.
While this approach provides explicit solution representations, it does not directly yield the $L^{1}$ estimates needed in our setting. Indeed, the nonlinearity in \eqref{eq:a-eq} contains a nonlocal integral term, and closing the Duhamel formulation requires uniform $L^{1}$ control of the associated half-line biharmonic heat kernel.

For this reason, we work instead with a direct spectral decomposition of the kernel. More precisely, we derive several estimates, including uniform $L^{1}$ bounds for the clamped biharmonic heat kernel and its derivatives. Also we show that, under suitable compatibility conditions on the initial data, the kernel preserves spatial regularity by using an alternative kernel representation obtained via integration by parts. These estimates allow us to establish local well-posedness of \eqref{eq:a-eq}. Furthermore, we prove the existence of finite-time blow-up solutions.
\endgroup

\section{Main Results}
Throughout this paper, $C^m(\mathbb{R}_+)$ denotes the space of functions whose
derivatives up to order $m$ are continuous and bounded on $\mathbb{R}_+$.
Moreover, $W^{m,1}(\mathbb{R}_+)$ denotes the Sobolev space of functions $f\in L^1(\mathbb{R}_+)$
such that $\partial_x^k f\in L^1(\mathbb{R}_+)$ for all integers $0\le k\le m$.
\subsection{Heat Kernel Representation}
\newcommand{\Rpp}{\mathbb{R}_+}
\newcommand{\PhiK}{\Phi}
\begin{prop}\label{prop:halfline_kernel}
Let $f\in C(\mathbb{R}_+)\cap L^1(\mathbb{R}_+)$. Consider the clamped biharmonic heat equation on the half-line:
\begin{equation}\label{eq:prob}
\begin{cases}
u_t(t,x) + u_{xxxx}(t,x)=0, & t>0,\; x>0,\\[1mm]
u(t,0)=0,\quad u_x(t,0)=0, & t>0,\\[1mm]
u(0,x)=f(x), & x>0.
\end{cases}
\end{equation}
We seek a classical solution $u$ such that, for every $t>0$ and every integer $m\ge0$,
$\lim_{x\to\infty} \partial_x^{\,m} u(t,x)=0.$

Define, for $t>0$ and $x,y\ge 0$,
\begin{equation}\label{eq:Phi_def}
\PhiK(r):=e^{-r}+\sin r-\cos r,
\end{equation}
and the kernel $K$,
\begin{equation}\label{eq:K_def}
K(t,x,y)
:=\frac{1}{\pi}\int_{0}^{\infty} e^{-k^{4}t}\,
\PhiK(kx)\,\PhiK(ky)\,dk.
\end{equation}
Set
\begin{equation}\label{eq:u_rep}
u(t,x):=\int_{0}^{\infty} K(t,x,y)\,f(y)\,dy.
\end{equation}
Then $u$ solves \eqref{eq:prob}.
\end{prop}
\begin{proof}
\medskip
\begin{enumerate}
\item[(i)]
We compute
\begin{align}
(u(t,x))_{t}
&= \int_{0}^{\infty}\int_{0}^{\infty}
-k^{4}e^{-k^{4}t}
\PhiK(kx)\,\PhiK(ky)\,dk\; f(y)\,dy \nonumber
\\
&=-(u(t,x))_{xxxx}.
\label{eq:proof_i_ut}
\end{align}
Thus $u$ satisfies the equation.

\item[(ii)]
We have
\begin{equation}\label{eq:proof_ii_bcK}
K_x(t,0,y)=0,\qquad K(t,0,y)=0.
\end{equation}
Hence, for each $t>0$,
\begin{equation}\label{eq:proof_ii_bcu}
u(t,0)=0,\qquad u_x(t,0)=0.
\end{equation}
Thus $u$ satisfies the boundary condition.

\item[(iii)]
For each $x>0$,
\begin{equation}\label{eq:proof_iv_main}
\lim_{t\to 0^+}\int_{0}^{\infty}K(t,x,y)\,f(y)\,dy=f(x).
\end{equation}
Fix $\varepsilon>0$. We write
\begin{align}
\int_{0}^{\infty}K(t,x,y)f(y)\,dy-f(x)
&= \int_{y-x \,\le\, \delta}K(t,x,y)\bigl(f(y)-f(x)\bigr)\,dy
\nonumber\\
&\quad + \int_{y-x\,>\,\delta}K(t,x,y)\bigl(f(y)-f(x)\bigr)\,dy
\nonumber\\
&\quad + \Bigl(\int_{0}^{\infty}K(t,x,y)\,dy-1\Bigr)f(x).
\label{eq:proof_iv_split3}
\end{align}
Using $|f(y)-f(x)|\le \varepsilon$ on $|y-x|\le\delta$, we obtain
\begin{align}
\Bigl|\int_{0}^{\infty}K(t,x,y)f(y)\,dy-f(x)\Bigr|
&\le \epsilon\int_{0}^{\infty}|K(t,x,y)|\,dy
\nonumber\\
&\quad + 2\|f\|_{L^\infty}\int_{\,y-x\,>\delta}|K(t,x,y)|\,dy
\nonumber\\
&\quad + \Bigl|\int_{0}^{\infty}K(t,x,y)\,dy-1\Bigr|\;|f(x)|.
\label{eq:proof_iv_est3}
\end{align}
 The uniform bound $\int_0^\infty |K(t,x,y)|\,dy \le C_0$ is consequence of Lemma~\ref{lem:L1bound_dxK}.
 
 By the corollary~\ref{cor:mass_one}, $\lim_{t\to0^+}\int_0^\infty K(t,x,y)\,dy=1$. Lastly, by the corollary~\ref{cor:mass_consentration},
$\lim_{t\to0^+}\int_{\,y-x\,>\delta}|K(t,x,y)|\,dy=0$. Thus we conclude that

\begin{equation}\label{eq:proof_iv_conclude_eps}
\lim_{t\to0^+}
\Bigl|\int_{0}^{\infty}K(t,x,y)f(y)\,dy-f(x)\Bigr|
\le \varepsilon C_0.
\end{equation}

Since $\varepsilon>0$ is arbitrary, \eqref{eq:proof_iv_main} follows.

\item[(iv)]
For each $t>0$ and each integer $m\ge 0$,
\begin{equation}\label{eq:proof_iii_limit_statement}
\lim_{x\to\infty} u_{x^{(m)}}(t,x)=0 .
\end{equation}
Indeed,
\begin{equation}\label{eq:proof_iii_bound_by_kernel}
\bigl|u_{x^{(m)}}(t,x)\bigr|
\le \int_{0}^{\infty}\bigl|\partial_x^{\,m}K(t,x,y)\bigr|\;|f(y)|\,dy .
\end{equation}
Moreover, By the lemma~\ref{lem:Linfbound_dxK}, 
\begin{equation}\label{eq:proof_iii_kernel_bound}
\bigl|\partial_x^{\,m}K(t,x,y)\,f(y)\bigr|
\le C\,t^{-(m+1)/4}\,\bigl|f(y)\bigr|.
\end{equation}

Thus, by the dominated convergence theorem,
\begin{equation}\label{eq:proof_iii_DCT}
\lim_{x\to\infty} u_{x^{(m)}}(t,x)
= \int_{0}^{\infty}\lim_{x\to\infty}\partial_x^{\,m}K(t,x,y)\,f(y)\,dy .
\end{equation}
By the Riemann--Lebesgue lemma, for each $t>0$,
\begin{equation}\label{eq:proof_iii_RL}
\lim_{x\to\infty}\partial_x^{\,m}K(t,x,y)=0.
\end{equation}
Combining \eqref{eq:proof_iii_DCT} and \eqref{eq:proof_iii_RL} yields \eqref{eq:proof_iii_limit_statement}.

\end{enumerate}
\end{proof}
\begin{lemma}\label{lem:L1bound_dxK}
For each integer $m\ge 0$, there exists a constant $C_m>0$, 
independent of $t>0$ and $x\ge0$, such that
\begin{equation}\label{eq:L1bound_statement}
\int_{0}^{\infty}\bigl|\partial_x^{\,m}K(t,x,y)\bigr|\,dy
\;\le\; C_m\,t^{-m/4}.
\end{equation}
\end{lemma}

\begin{proof}
Differentiate under the $k$-integral. Since $\partial_x^m \PhiK(kx)=k^m\PhiK^{(m)}(kx)$,
\begin{equation}\label{eq:dxmK_before_scaling}
\partial_x^{\,m}K(t,x,y)
=\frac1\pi\int_{0}^{\infty} k^{m}e^{-k^{4}t}\,\PhiK^{(m)}(kx)\,\PhiK(ky)\,dk.
\end{equation}
Write $k^m=t^{-m/4}(kt^{1/4})^m$ and define
\begin{equation}\label{eq:gm_def}
g_m(t,x,y):=\int_{0}^{\infty} (kt^{1/4})^{m}e^{-k^{4}t}\,\PhiK^{(m)}(kx)\,\PhiK(ky)\,dk.
\end{equation}
Then
\begin{align}
\int_{0}^{\infty}\bigl|\partial_x^{\,m}K(t,x,y)\bigr|\,dy
&=\frac{t^{-m/4}}{\pi}\int_{0}^{\infty}|g_m(t,x,y)|\,dy. \label{eq:L1reduce_gm}
\end{align}
Hence it suffices to prove
\begin{equation}\label{eq:goal_uniform_gm}
\sup_{t>0,\ x\ge 0}\ \int_{0}^{\infty}|g_m(t,x,y)|\,dy \;\le\;C_m.
\end{equation}

\medskip
\noindent\textbf{Step 1 (the case $x=0$).} We show that, for $x=0$,
\begin{equation}\label{eq:step1_goal}
\int_{0}^{\infty}\bigl|g_m(t,0,y)\bigr|\,dy \le C_m .
\end{equation}

We split the $y$-integral into $(0,1)$ and $(1,\infty)$.

\medskip
\noindent\underline{(i) The part $0<y<t^{\frac14}$.}
With the change of variables $s=k\,t^{1/4}$ and $z=y\,t^{-1/4}$, we obtain:

\begin{equation}
\begin{aligned}
\int_{0}^{1}\bigl|g_m(t,0,y)\bigr|\,dy
&= \int_{0}^{1}\Bigl|\int_{0}^{\infty} As^{m}e^{-s^{4}}
\PhiK(sz)\,ds\Bigr|\,dz,  \\
&\le \int_{0}^{1}\int_{0}^{\infty} 3A\,s^{m}e^{-s^{4}}\,ds\,dz. \label{eq:step1_small_2}
\end{aligned}
\end{equation}

\medskip
\noindent\underline{(ii) The part $y>t^{\frac14}$.}
Again using the same $s$ and $z$, we write
\begin{equation}\label{eq:step1_large_1}
\int_{1}^{\infty}\bigl|g_m(t,0,y)\bigr|\,dy
= \int_{1}^{\infty}\Bigl|\int_{0}^{\infty} As^{m}e^{-s^{4}}
\PhiK(sz)\,ds\Bigr|\,dz .
\end{equation}
Integrating by parts twice in $s$, we get
\begin{equation}
\begin{aligned}
\int_{1}^{\infty}\bigl|g_m(t,0,y)\bigr|\,dy
&= \int_{1}^{\infty}\frac{A}{z^{2}}
\Bigl|\int_{0}^{\infty}\bigl(e^{-sz}-\sin(sz)+\cos(sz)\bigr)\,
\bigl(s^{m}e^{-s^{4}}\bigr)''\,ds\Bigr|\,dz,\\
&\le \int_{1}^{\infty}\frac{A}{z^{2}}
\int_{0}^{\infty} 3\,p_m(s)\,e^{-s^{4}}\,ds\,dz. \label{eq:step1_large_3}
\end{aligned}
\end{equation}
where $p_m$ is a polynomial depending only on $m$.
Combining \eqref{eq:step1_small_2} and \eqref{eq:step1_large_3} yields \eqref{eq:step1_goal}.

\medskip
\noindent\textbf{Step 2 (the case $t\rightarrow0^+$).} We show that, for each $x>0$,
\begin{equation}\label{eq:step2_goal}
\lim_{t\to 0^+}\int_{0}^{\infty}\bigl|g_m(t,x,y)\bigr|\,dy \le C_m .
\end{equation}
Expanding the product inside $g_m$, we decompose
\begin{equation}\label{eq:ABC_decomp}
\int_{0}^{\infty}\bigl|g_m(t,x,y)\bigr|\,dy
\le \int_{0}^{\infty}\Bigl(|A_m(t,x,y)|+|B_m(t,x,y)|+|C_m(t,x,y)|\Bigr)\,dy,
\end{equation}
where
\begin{align}
A_m(t,x,y)
&:= \int_{0}^{\infty}(k t^{1/4})^{m}e^{-k^{4}t}
\bigl((-1)^me^{-k(x+y)}-\sin(k(x+y)+\tfrac{m\pi}{2})\bigr)\,dk,
\label{eq:A_def} \\[6pt]
B_m(t,x,y)
&:= \int_{0}^{\infty}(k t^{1/4})^{m}e^{-k^{4}t}
\cos\bigl(k(x-y)+\tfrac{m\pi}{2}\bigr)\,dk,
\label{eq:B_def} \\[6pt]
C_m(t,x,y)
&:= \int_{0}^{\infty}(k t^{1/4})^{m}e^{-k^{4}t}
\Bigl(e^{-ky}(\sin(kx+\tfrac{m\pi}{2})-\cos(kx+\tfrac{m\pi}{2})) \nonumber\\
&\qquad\qquad
+(-1)^me^{-kx}(\sin(ky)-\cos(ky))\Bigr)\,dk.
\label{eq:C_def}
\end{align}

\noindent\underline{Estimate for $A$.}
By employing the change of variables $s = k t^{1/4}$ and $z = (x+y)t^{-1/4}$ into \eqref{eq:A_def},
we obtain:
\begin{align}
&\lim_{t\to 0^+}\int_{0}^{\infty}|A_m(t,x,y)|\,dy\nonumber \\[6pt] 
&\qquad= \lim_{t\to 0^+}\int_{x/t^{1/4}}^{\infty}
\Bigl|\int_{0}^{\infty}s^{m}e^{-s^{4}}\bigl(e^{-zs}-\sin(zs+\frac{m\pi}2)\bigr)\,ds\Bigr|\,dz,\label{eq:A_limit_2} \\[6pt] 
&\qquad\le \lim_{t\to 0^+}\int_{x/t^{1/4}}^{\infty}\int_{0}^{\infty}
\frac{2}{z^{2}}\,p_m(s)\,e^{-s^{4}}\,ds\,dz
=0. \nonumber 
\end{align}
Hence $\lim_{t\to 0^+}\int_{0}^{\infty}|A_m(t,x,y)|\,dy=0$.

\medskip
\noindent\underline{Estimate for $B$.} By employing the change of variables $s = k t^{1/4}$ and $z = (y-x)t^{-1/4}$ into \eqref{eq:B_def},
we obtain:
\begin{align}
\int_{0}^{\infty}|B_m(t,x,y)|\,dy
&=\int_{-\frac{x}{t^{1/4}}}^{\infty}
\Bigl|\int_{0}^{\infty}s^{m}e^{-s^{4}}\cos(sz+\frac{m\pi}2)\,ds\Bigr|\,dz .
\label{eq:B_L1_rewrite}
\end{align}
Therefore,
\begin{align}
\lim_{t\to 0^+}\int_{0}^{\infty}|B_m(t,x,y)|\,dy
&=\int_{-\infty}^{\infty}
\Bigl|\int_{0}^{\infty}s^{m}e^{-s^{4}}\cos(sz+\frac{m\pi}2)\,ds\Bigr|\,dz
\le C_m .
\label{eq:B_L1_limit}
\end{align}

\medskip
\noindent\underline{Estimate for $C$.} Integrating by parts twice in $k$ for the oscillatory terms in \eqref{eq:C_def}, we write:
\begin{align}
&|C_m(t,x,y)|
\le \Bigl|t^{m/4}\int_{0}^{\infty}\frac{1}{y^{2}}
\bigl(-\sin(ky)+\cos(ky)\bigr)\,\Bigl(k^{m}e^{-k^{4}t}(-1)^me^{-kx}\Bigr)''\,dk\Bigr|
\nonumber\\[6pt]
&\quad+\Bigl|t^{m/4}\int_{0}^{\infty}\frac{1}{y^{2}}
\bigl(-e^{-ky}\bigr)\,\Bigl(k^{m}\bigl(e^{-k^{4}t}(\sin(kx +\frac{m\pi}2)-\cos(kx+\frac{m\pi}2)\bigr)\Bigr)''\,dk\Bigr|.\label{eq:C_ibp_2}
\end{align}
Hence, for $y\ge 1$,
\begin{equation}\label{eq:C_y_ge_1}
|C_m(t,x,y)|\le \frac{F_1(x)}{y^{2}},
\qquad\text{uniformly for } t\le T,\ x>0,
\end{equation}
and therefore $C(t,x,\cdot)\in L^{1}([1,\infty))$ uniformly in $t\le T$.

\medskip
\noindent For $0<y\le 1$, we use the (one-step) integration-by-parts decomposition:
\begin{align}
&|C_m(t,x,y)|
\le t^{m/4}\Bigl|\int_{0}^{\infty} C\,k^{m}e^{-kx}\,dk\Bigr|\nonumber \\[6pt]
&\quad + t^{m/4}\Bigl|\int_{0}^{\infty}\frac{1}{x}\,(-\cos(kx +\frac{m\pi}2)-\sin(kx +\frac{m\pi}2))\,\Bigl(k^{m}e^{-k^{4}t}e^{-ky}\Bigr)'\,dk\Bigr|.
\end{align}

and consequently,
\begin{equation}
\begin{aligned}
|C_m(t,x,y)|
&\le \frac{C_{1}}{x^{m+1}}
+\Bigl|\int_{0}^{\infty}\frac{2}{x}\,y e^{-ky}(k t^{1/4})^{m}e^{-k^{4}t}\,dk\Bigr|\\
&+\Bigl|\int_{0}^{\infty}\frac{C_1'}{x}(k t^{1/4})^{m-1}e^{-k^{4}t}\,t^{1/4}\,dk\Bigr|
\\
&+\Bigl|\int_{0}^{\infty}\frac{C_2'}{x}(k t^{1/4})^{m+3}e^{-k^{4}t}\,t^{1/4}\,dk\Bigr|
\le\frac{C_{1}}{x^{m+1}}+\frac{C_{2}}{x}.
\label{eq:C_small_y_bound_2}
\end{aligned}
\end{equation}
uniformly for $t\le T,\ x>0$.
Therefore, $C(t,x,\cdot)\in L^{1}([0,1])$ uniformly in $t\le T$.

\medskip
\noindent Finally, by the dominated convergence theorem,
\begin{equation}\label{eq:C_DCT_y}
\lim_{t\to 0^+}\int_{0}^{\infty}|C_m(t,x,y)|\,dy
=\int_{0}^{\infty}\lim_{t\to 0^+}|C_m(t,x,y)|\,dy .
\end{equation}
Moreover, a dominating bound of the type
\begin{equation}\label{eq:C_dominate_sketch}
|C_m(t,x,y)|
\le \int_{0}^{\infty} A\,k^{m}\bigl(e^{-ky}+e^{-kx}\bigr)\,dk
\end{equation}
ensures that we may also pass the limit inside the $k$-integral:
\begin{equation}
\begin{aligned}
&\lim_{t\to 0^+}|C_m(t,x,y)|  \\
&= 
\int_{0}^{\infty}\lim_{t\to 0^+}(k t^{1/4})^{m}e^{-k^{4}t}
\Bigl(e^{-ky}(\sin(kx+\frac{m\pi}2)-\cos(kx+\frac{m\pi}2)) \label{eq:C_DCT_k} \\[6pt]
&\qquad+e^{-kx}(\sin(ky)-\cos(ky))\Bigr)\,dk .
\end{aligned}
\end{equation}

If $m>0$, then $\lim_{t\to 0^+}(k t^{1/4})^{m}e^{-k^{4}t}=0$ for each fixed $k>0$, 

If $m=0$, then
\begin{equation}\label{eq:C_limit_mzero_integral}
\int_{0}^{\infty}\Bigl(e^{-ky}(\sin(kx)-\cos(kx))+e^{-kx}(\sin(ky)-\cos(ky))\Bigr)\,dk=0,
\end{equation}
thus $\lim_{t\to 0^+}\int_{0}^{\infty}|C_m(t,x,y)|\,dy=0$.
Therefore, combining the estimates for $A_m, B_m$ and $C_m$, we obtain \eqref{eq:step2_goal}.

\medskip
\noindent\textbf{Step 3.} For each fixed $x = x_0 >0$, we show that
\begin{equation}\label{eq:step3_goal}
\int_{0}^{\infty}\bigl|g_m(t,x_0,y)\bigr|\,dy \le C_m
\qquad\text{for } t\ge t_0 .
\end{equation}
Let $s=t^{1/4}k$ and $z=y\,t^{-1/4}$. Then we can rewrite the tail part ($y\ge t^{\frac14}$) of the \eqref{eq:step3_goal} as
\begin{equation}
\int_{1}^{\infty}\bigl|g_m(t,x_0,y)\bigr|\,dy \nonumber\\
= \int_{1}^{\infty}\Bigl|\int_{0}^{\infty} s^{m}e^{-s^{4}}
\PhiK^{(m)}(\frac{sx_0}{t^{1/4}})\,
\PhiK(sz)\,ds\Bigr|\,dz .
\label{eq:step3_tail_rewrite}
\end{equation}
Applying integration by parts twice in $s$ (as in Step 1), we obtain a bound of the form
\begin{align}
\int_{t^{\frac14}}^{\infty}\bigl|g_m(t,x_0,y)\bigr|\,dy
&= \int_{1}^{\infty}\frac{1}{z^{2}}
\Bigl(C_1\Bigl(\frac{x_0}{t^{1/4}}\Bigr)^{2}
+ C_2\Bigl(\frac{x_0}{t^{1/4}}\Bigr)+C_3\Bigr)\,dz
\notag\\
&\le C_1\Bigl(\frac{x_0}{t_0^{1/4}}\Bigr)^{2}
+ C_2\Bigl(\frac{x_0}{t_0^{1/4}}\Bigr)+C_3,
\label{eq:step3_tail_bound}
\end{align}
where the last inequality uses $t\ge t_0>0$ and fixed $x>0$.

For the near part ($0\le y\le t^{\frac14}$), we use the crude bound $|\Phi(\cdot)|\le 3$ to get
\begin{align}
\int_{0}^{t^{\frac14}}\bigl|g_m(t,x_0,y)\bigr|\,dy
&\le \int_{0}^{1}\int_{0}^{\infty} B\,s^{m}e^{-s^{4}}\,ds\,dz.
\label{eq:step3_near_bound}
\end{align}
Combining \eqref{eq:step3_tail_bound} and \eqref{eq:step3_near_bound} yields \eqref{eq:step3_goal}.

\medskip
\noindent\textbf{Step 4.} For each fixed $x = x_0>0$, we have
\begin{equation}\label{eq:step4_goal}
\int_{0}^{\infty}\bigl|g_m(t,x_0,y)\bigr|\,dy \le C_m
\qquad \text{for all } t>0.
\end{equation}
Moreover, the map
\begin{equation}\label{eq:step4_continuity}
t \longmapsto \int_{0}^{\infty}\bigl|g_m(t,x_0,y)\bigr|\,dy
\end{equation}
is continuous on $(0,\infty)$, and thus \eqref{eq:step4_goal} follows from Step~2 and Step~3.

\medskip
\noindent\textbf{Step 5.}
Define $\widetilde{g}_m\!\left(\frac{x}{t^{1/4}},\,z\right):= \int_{0}^{\infty} s^{m}e^{-s^{4}}
\PhiK^{(m)}(\frac{sx_0}{t^{1/4}})\,
\PhiK(sz)\,ds$. And we use this scaling representation to write
\begin{equation}\label{eq:step5_scaling}
\int_{0}^{\infty}\bigl|g_m(t,x,y)\bigr|\,dy
= \int_{0}^{\infty}\bigl|\widetilde{g}_m\!\left(\frac{x}{t^{1/4}},\,z\right)\bigr|\,dz .
\end{equation}
Thus,
\begin{equation}\label{eq:step5_conclusion}
\int_{0}^{\infty}\bigl|g_m(t,x,y)\bigr|\,dy \le C_m
\qquad\text{for all } x>0,\ t>0.
\end{equation}
Together with Step~1, this completes the proof of the lemma.
\end{proof}

 We can derive the following corollaries. Corollary~\ref{cor:mass_one} and \ref{cor:mass_consentration} follow from Step~2 in the proof of Lemma~1.
\medskip
\medskip
\begin{corollary}\label{cor:mass_one}
For each $x>0$,
\begin{equation}\label{eq:mass_one}
\lim_{t\to 0^+}\int_{0}^{\infty} K(t,x,y)\,dy = 1 .
\end{equation}
\end{corollary}
\begin{proof}
As shown in Step~2 of the Lemma~1, the integral of $A_0$ and $C_0$ vanish when $t \rightarrow 0^+$, and $\lim_{t\to 0^+}\int_0 ^{\infty} B_0\,dy = 1$. Thus, $\lim_{t\to 0^+}\int_0 ^{\infty} K\,dy \,=\,\lim_{t\to 0^+}\int_0 ^{\infty} (A_0 + B_0+C_0)\,dy \, = 1$.
\end{proof}

\begin{corollary}\label{cor:mass_consentration}
For each $x,\, \delta>0$,
\begin{equation}\label{eq:mass_one2}
\lim_{t\to0^+}\int_{\,y-x\,>\delta}|K(t,x,y)|\,dy=0.
\end{equation}
\end{corollary}
\begin{proof}
As in Step~2 of Lemma~1 (cf.\ the argument in Lemma~2), we treat the terms
$A_0,B_0,C_0$ separately.

For the $A_0$-term, using the change of variables $s=k\,t^{1/4}$ and
$z=(x+y)/t^{1/4}$, we obtain
\begin{equation}
\lim_{t\to 0^+}\int_{y-x>\delta} |A_0|\,dy
\le
\lim_{t\to 0^+}\int_{(2x+\delta)/t^{1/4}}^{\infty}\int_{0}^{\infty}
\frac{2}{z^{2}}\,p_m(s)\,e^{-s^{4}}\,ds\,dz
=0.
\end{equation}

For the $B_0$-term, with the change of variables $s=k\,t^{1/4}$ and
$z=(x-y)/t^{1/4}$, we have
\begin{equation}
\lim_{t\to 0^+}\int_{y-x>\delta} |B_0|\,dy
\le
\lim_{t\to 0^+}\int_{\delta/t^{1/4}}^{\infty}
\left|\int_{0}^{\infty} e^{-s^{4}}\cos(sz)\,ds\right|\,dz
=0.
\end{equation}

Finally, $C_0\to 0$ as $t\to 0^+$. Hence, by the dominated convergence
argument used in Step~2 of Lemma~1,
\begin{equation}
\lim_{t\to 0^+}\int_{y-x>\delta} |C_0|\,dy = 0.
\end{equation}
This completes the proof of the corollary.
\end{proof}

\medskip

\begin{lemma}\label{lem:Linfbound_dxK}
For each integer $m\ge 0$, there exists a constant $C'_m>0$, 
independent of $t>0$ and $x\ge0$, such that
\begin{equation}\label{eq:L1bound_statement2}
\bigl|\partial_x^{\,m}K(t,x,y)\bigr| 
\;\le\; C'_m\,t^{-(m+1)/4}.
\end{equation} 
\end{lemma}
\begin{proof}
 The function $g_m$ satisfies
\begin{equation}
 g_m(t,x,y) = t^{-1/4}\widetilde{g}_m\left(\frac{x}{t^{1/4}}, z\right), \label{eq:scaling_gm}
\end{equation}

Therefore, the profile function $\widetilde{g}_m$ is uniformly bounded:
\begin{equation}
\bigl|\widetilde{g}_m\left(\frac{x}{t^{1/4}}, z\right)\bigr| \le  M \int_0^\infty s^m e^{-s^4} ds =: C'_m.
\end{equation}
Substituting this into \eqref{eq:L1bound_statement2}, we obtain the desired estimate. And
this completes the proof.
\end{proof}

The kernel $K(t,x,y)$ we use is not translation-invariant in the sense that it does not depend only on $x-y$.
Therefore, in proving $L^1$ boundedness (see Corollary~\ref{cor:Estimation10}) or the preservation of regularity (see Corollary~\ref{cor:idff}),
we cannot rely on the standard argument for convolution kernels.
Instead, we perform integration by parts and use the uniform boundedness of the modified kernels
$K_a$, $K_b$, and $K_c$ whose definitions are given in Lemma~\ref{lem:3}.

\begin{lemma}\label{lem:3}
For $r \ge 0$, let
\begin{equation}\label{eq:Phi_defs}
\begin{aligned}
\Phi_1(r) &:= e^{-r}-\sin r-\cos r,\\
\Phi_2(r) &:= e^{-r}-\sin r+\cos r,\\
\Phi_3(r) &:= e^{-r}+\sin r+\cos r.
\end{aligned}
\end{equation}

For $t>0$, $x\ge0$, $y\ge0$, define the kernels
\begin{equation}\label{eq:Kabc_defs}
\begin{aligned}
K_a(t,x,y)
&:= \frac{1}{\pi}\int_{0}^{\infty} e^{-k^{4}t}\,
   \Phi_3(kx)\,\Phi_1(ky)\,dk,\\
K_b(t,x,y)
&:= \frac{1}{\pi}\int_{0}^{\infty} e^{-k^{4}t}\,
   \Phi_2(kx)\,\Phi_2(ky)\,dk,\\
K_c(t,x,y)
&:= \frac{1}{\pi}\int_{0}^{\infty} e^{-k^{4}t}\,
   \Phi_1(kx)\,\Phi_3(ky)\,dk.
\end{aligned}
\end{equation}
For each integer $m\ge 0$, there exist constants $C_{a,m},\, C_{b,m},\,C_{c,m}>0$,
independent of $t>0$ and $x\ge0$, such that
\begin{equation}\label{eq:3-3}
\begin{aligned}
\int_{0}^{\infty} \bigl|\partial_x^{m} K_a(t,x,y)\bigr|\,dy
&\le  C_{a,m}\,t^{-m/4}, \\[4pt]
\int_{0}^{\infty} \bigl|\partial_x^{m} K_b(t,x,y)\bigr|\,dy
&\le  C_{b,m}\,t^{-m/4}, \\[4pt]
\int_{0}^{\infty} \bigl|\partial_x^{m} K_c(t,x,y)\bigr|\,dy
&\le  C_{c,m}\,t^{-m/4}.
\end{aligned}
\end{equation}

\end{lemma}
\begin{proof}
The argument follows Lemma 1.
\end{proof}

\begin{corollary}[$L^1$, $L^\infty$ bounds, and smoothing]\label{cor:Estimation10}

Define:
\begin{equation}\label{eq:semigroup_def}
(S(t)f)(x):=\int_{0}^{\infty} K(t,x,y)\,f(y)\,dy,\qquad t>0,\ x>0 .
\end{equation}

For each integer $m\ge 0$ there exist constants $D_m,\, C_m ,\, C'_m>0$ such that for all $f\in L^{1}(\mathbb{R}_+) \cap  C(\mathbb{R}_+)$ and $t>0$,
\begin{align}
\|\partial_x^{\,m} S(t)f\|_{L^1(\mathbb{R}_+)} \nonumber
&\le D_m\, t^{{-m /4}}
\|f\|_{L^1(\mathbb{R}_+)}, \\[6pt] \label{eq:smoothing}
\|\partial_x^{\,m} S(t)f\|_{C(\mathbb{R}_+)}
&\le C_m\, t^{-m /4}
\|f\|_{C(\mathbb{R}_+)},\\[6pt] \nonumber
\|\partial_x^{\,m} S(t)f\|_{C(\mathbb{R}_+)}
&\le C'_m\, t^{- (m+1)/4}
\|f\|_{L^1(\mathbb{R}_+)}.
\end{align}
\end{corollary}

\begin{proof}

\smallskip
\noindent\emph{$L^1$-smoothing from $L^1$.}
By Tonelli's theorem,
\begin{equation} \label{eq:loco}
\int_0^\infty\Big|\int_0^\infty \partial_x^{\,m}K(t,x,y)f(y)\,dy\Big|dx\\
\le \int_0^\infty |f(y)|\Big(\int_0^\infty |\partial_x^{\,m}K(t,x,y)|\,dx\Big)dy.
\end{equation}

By the symmetry property of the kernel $K$,
\begin{equation} \label{eq:symm}
\sup_{y\ge0}\int_0^\infty |\partial_x^{\,m}K(t,x,y)|\,dx
=\sup_{x\ge0}\int_0^\infty |\partial_y^{\,m}K(t,x,y)|\,dy.
\end{equation}

The derivatives of the kernel $K$ with respect to $y$ can be effectively replaced by the $x$-derivatives of the modified kernels $K_a, K_b$ and $K_c$. Specifically, for any $j \in \mathbb{N}_0$, we have the following relations:\begin{equation}\label{eq:kernel_relations}
\begin{aligned}
\partial_y^{4j} K(t,x,y)
&= \partial_x^{4j} K(t,x,y), \\[6pt]
\partial_y^{4j+1} K(t,x,y)
&= \partial_x^{4j+1} K_a(t,x,y),\\[6pt]
\partial_y^{4j+2} K(t,x,y)
&= \partial_x^{4j+2} K_b(t,x,y), \\[6pt]
\partial_y^{4j+3} K(t,x,y)
&= \partial_x^{4j+3} K_c(t,x,y).
\end{aligned}
\end{equation}

Applying \eqref{eq:kernel_relations} into \eqref{eq:loco}, and Lemma~\ref{lem:L1bound_dxK} and \ref{lem:3} yield:
\begin{equation} 
\int_0^\infty |\partial_x^{\,m}K(t,x,y)|\,dx \le D_m\,t^{-m/4}
\qquad \text{for all } y\ge0.
\end{equation}
Therefore,
\begin{equation} \label{eq:l234}
\|\partial_x^{\,m}S(t)f\|_{L^1(\mathbb{R}_+)}\le D_m\,t^{-m/4}\|f\|_{L^1(\mathbb{R}_+)}.
\end{equation}

\noindent\emph{$C$-smoothing from $C$.}
 By Lemma~\ref{lem:L1bound_dxK},
\begin{equation}
|\partial_x^{\,m}S(t)f|
\le \|f\|_{C(\mathbb{R}_+)}\int_0^\infty |\partial_x^{\,m}K(t,x,y)|\,dy
\le C_m\,t^{-m/4}\|f\|_{C(\mathbb{R}_+)}.
\end{equation}

\smallskip
\noindent\emph{$C$-smoothing from $L^1$.}
By Lemma~\ref{lem:Linfbound_dxK},
\begin{equation}
|\partial_x^{\,m}S(t)f|
\le \int_0^\infty |\partial_x^{\,m}K(t,x,y)|\,|f(y)|\,dy
\le C'_m\,t^{-(m+1)/4}\|f\|_{L^1(\mathbb{R}_+)}.
\end{equation}
This completes the proof.
\end{proof}

\begin{corollary}[Regularity preservation]\label{cor:idff}
Let $S(t)$ be defined by \eqref{eq:semigroup_def}. Fix $m\ge 0$.
Assume the compatibility condition:
\begin{equation}\label{eq:compat_4m}
\partial_x^{4j}f(0)=0\ (j=0,\dots,m),\qquad 
\partial_x^{4j+1}f(0)=0\ (j=0,\dots,m-1),
\end{equation}
and, when stated, the additional condition
\begin{equation}\label{eq:compat_4m_plus1}
\partial_x^{4m+1}f(0)=0.
\end{equation}

There exists $C>0$, independent of $t>0$, such that the following hold:

\medskip
\noindent\emph{(i) If $f\in C^{4m+1}(\mathbb R_+)\cap W^{4m+1,1}(\mathbb R_+)$ and \eqref{eq:compat_4m} holds, then}
\begin{align}
\|\partial_x^{4m+1}S(t)f\|_{C(\mathbb R_+)} \le C\|f\|_{C^{4m+1}(\mathbb R_+)},\,
\|\partial_x^{4m+1}S(t)f\|_{L^1(\mathbb R_+)} \le C\|f\|_{W^{4m+1,1}(\mathbb R_+)}. \label{eq:reg_4m+1}
\end{align}

\medskip
\noindent\emph{(ii) If $f\in C^{4m+2}(\mathbb R_+)\cap W^{4m+2,1}(\mathbb R_+)$ and \eqref{eq:compat_4m}--\eqref{eq:compat_4m_plus1} hold, then}
\begin{align}
\|\partial_x^{4m+2}S(t)f\|_{C(\mathbb R_+)} \le C\|f\|_{C^{4m+2}(\mathbb R_+)},\,
\|\partial_x^{4m+2}S(t)f\|_{L^1(\mathbb R_+)} \le C\|f\|_{W^{4m+2,1}(\mathbb R_+)}. \label{eq:reg_4m+2}
\end{align}

\medskip
\noindent\emph{(iii) If $f\in C^{4m+3}(\mathbb R_+)\cap W^{4m+3,1}(\mathbb R_+)$ and \eqref{eq:compat_4m}--\eqref{eq:compat_4m_plus1} hold, then}
\begin{align}
\|\partial_x^{4m+3}S(t)f\|_{C(\mathbb R_+)} \le C\|f\|_{C^{4m+3}(\mathbb R_+)},\,
\|\partial_x^{4m+3}S(t)f\|_{L^1(\mathbb R_+)} \le C\|f\|_{W^{4m+3,1}(\mathbb R_+)}. \label{eq:reg_4m+3}
\end{align}

\medskip
\noindent\emph{(iv) If $f\in C^{4m+4}(\mathbb R_+)\cap W^{4m+4,1}(\mathbb R_+)$ and \eqref{eq:compat_4m}--\eqref{eq:compat_4m_plus1} hold, then}
\begin{align}
\|\partial_x^{4m+4}S(t)f\|_{C(\mathbb R_+)} \le C\|f\|_{C^{4m+4}(\mathbb R_+)},\,
\|\partial_x^{4m+4}S(t)f\|_{L^1(\mathbb R_+)} \le C\|f\|_{W^{4m+4,1}(\mathbb R_+)}. \label{eq:reg_4m+4}
\end{align}
\end{corollary}

\begin{proof}
We differentiate under the integral sign. As in \eqref{eq:kernel_relations},
$x$-derivatives of the kernel can be expressed in terms of the corresponding
$y$-derivatives of $K$ and of the modified kernels $K_a$, $K_b$, and $K_c$.

\begin{equation}\label{eq:smoothing3}
\begin{aligned}
\partial_x^{4m+1}(S(t)f)(x)
  &= \int_{0}^{\infty} \partial_y^{4m+1}K_a(t,x,y)\,f(y)\,dy, \\
\partial_x^{4m+2}(S(t)f)(x)
  &= \int_{0}^{\infty} \partial_y^{4m+2}K_b(t,x,y)\,f(y)\,dy, \\
\partial_x^{4m+3}(S(t)f)(x)
  &= \int_{0}^{\infty} \partial_y^{4m+3}K_c(t,x,y)\,f(y)\,dy, \\
\partial_x^{4m+4}(S(t)f)(x)
  &= \int_{0}^{\infty} \partial_y^{4m+4}K(t,x,y)\,f(y)\,dy .
\end{aligned}
\end{equation}

We differentiate under the integral sign and repeatedly integrate by parts in $y$. As $y \to \infty$, the kernels $K$, $K_a$, $K_b$, and $K_c$,
as well as the kernel derivatives appearing in the integration-by-parts
procedure, decay to zero by the Riemann--Lebesgue lemma.

At the boundary $y=0$, the kernels and their derivatives do not necessarily vanish. Nevertheless, each boundary term produced by integration by parts contains a factor of the form $\partial_y^m f(0)$ for some $m \ge 0$. The compatibility conditions \eqref{eq:compat_4m}--\eqref{eq:compat_4m_plus1}
guarantee that all such boundary traces vanish, and therefore no boundary
contribution remains at $y=0$. 

Consequently, no boundary contribution appears in the integration-by-parts argument. Thus,
\begin{equation}\label{eq:smoothing3_ibp}
\begin{aligned}
\partial_x^{4m+1}(S(t)f)(x)
&= - \int_{0}^{\infty} {K_a}(t,x,y)\,\partial_y^{4m+1}f(y)\,dy, \\
\partial_x^{4m+2}(S(t)f)(x)
&= \int_{0}^{\infty} {K_b}(t,x,y)\,\partial_y^{4m+2}f(y)\,dy, \\
\partial_x^{4m+3}(S(t)f)(x)
&= -\int_{0}^{\infty} {K_c}(t,x,y)\,\partial_y^{4m+3}f(y)\,dy, \\
\partial_x^{4m+4}(S(t)f)(x)
&= \int_{0}^{\infty} K(t,x,y)\,\partial_y^{4m+4}f(y)\,dy .
\end{aligned}
\end{equation}

By Lemma~\ref{lem:L1bound_dxK} and \ref{lem:3}, we complete the proof.
\end{proof}




\begin{prop}[Uniqueness]\label{prop:halfline_kernel_unique}
Assume that $u,v$ are two solutions to \eqref{eq:prob} on $[0,T]$ such that
\begin{equation}
u,v\in C([0,T]; W^{2,1}(\mathbb R_+) \cap C^4(\mathbb R_+)),
\end{equation}
Then $u\equiv v$ on $[0,T]\times\mathbb R_+$.
\end{prop}
\begin{proof}

Suppose that there exist two solutions $u$ and $v$
to \eqref{eq:prob} on $[0,T]$ with same initial datum, 
Set $w := u - v$, then $w \in C\big([0,T];W^{2,1}(\mathbb R_+) \cap C^{4}(\mathbb{R}_+)\big)$, $w(0,x)=0$, and $w$ satisfies the clamped boundary conditions:\
\begin{equation}
w(t,0)=0, \quad w_x(t,0)=0,  \quad t\in[0,T].
\end{equation}
Moreover, $w(t,\cdot)$ and $w_x(t,\cdot)$ vanish at infinity for each $t > 0$.
Set the energy as follows, \begin{equation}
 E(t) = \int_0^\infty \frac{w^2  } {2}dx  .
 \end{equation}

Differentiate the $E(t)$ with respect to the $t$ (
$w(t,\cdot)\in L^1(\mathbb{R}_+)\cap L^\infty(\mathbb{R}_+)$ and
$w_{xxxx}(t,\cdot)\in L^\infty(\mathbb{R}_+)$),

 \begin{equation}
 \begin{aligned}
  \frac{dE(t)}{dt}
 =\int_0^\infty (w\, w_t )\,dx = - \int_0^\infty (w\, w_{xxxx} )\,dx.
 \label{eq:abc}
 \end{aligned}
 \end{equation}

 Since $w \in C\big([0,T];W^{2,1}(\mathbb R_+) \cap C^4(\mathbb{R}_+)\big)$, the integration by parts in \eqref{eq:abc} is justified. Furthermore, as all boundary terms vanish due to the clamped boundary conditions and decay properties, \eqref{eq:abc} reduces to:

  \begin{equation}
 \begin{aligned}
  \frac{dE(t)}{dt} = - \int_0^\infty (w_{xx})^2\,dx \le 0.
 \label{eq:abc2}
 \end{aligned}
 \end{equation}

 Thus, $E(t) \le E(0) = 0$ and $u = v$ on $[0,T]$.
 \end{proof}

\begin{corollary}[Semigroup property]\label{cor:Semigroup}

Let $f\in  L^{1}(\mathbb R_+) \cap C(\mathbb R_+)$. Then for any $\tau,s>0$ and $x>0$, operator $S(t)$ of \eqref{eq:semigroup_def} satisfies:
\begin{equation}\label{eq:semigroup_property}
S(\tau)\big(S(s)f\big)(x) = S(\tau+s)f(x).
\end{equation}
\end{corollary}

\begin{proof}
Fix $s>0$ and set
\begin{equation}\label{eq:sg_def}
g:=S(s)f .
\end{equation}
By the smoothing property of Corollary~\ref{cor:Estimation10}, we have $g\in  W^{2,1}(\mathbb R_+) \cap C_{4}(\mathbb R_+)$. 
Define:
\begin{equation}\label{eq:u_v_def}
\begin{aligned}
u(t,x)
&:= S(t)g(x) = S(t)\bigl(S(s)f\bigr)(x),\\[6pt]
v(t,x)
&:= S(t+s)f(x), \qquad t\in[0,\tau].
\end{aligned}
\end{equation}

Then $u$ and $v$ solve \eqref{eq:prob} on $[0,\tau]$ and share the same initial data at $t=0$:
\begin{equation}\label{eq:same_initial}
u(0,x)=g(x)=v(0,x).
\end{equation}
Moreover, by Lemma~\ref{lem:L1bound_dxK} and Corollary~\ref{cor:idff} (noting that $g$ satisfies the compatibility conditions, $g(0) = g_x(0) = 0$),
\begin{equation}\label{eq:u_v_reg}
u,v\in C([0,T];W^{2,1}(\mathbb R_+) \cap C_{4}(\mathbb R_+)).
\end{equation}

Therefore, Theorem~\ref{prop:halfline_kernel_unique} (uniqueness) yields
\begin{equation}\label{eq:semigroup_identity_local}
S(\tau)\big(S(s)f\big)(x)=S(\tau+s)f(x).
\end{equation}

\end{proof}

\subsection{Existence and Uniqueness of the Solution}
\begin{prop}\label{prop:exe}
Consider the restricted fourth-order Prandtl equation \eqref{eq:a-eq}
\begin{equation}
a_t = - a_{yyyy} +a^2 - a_y\int_0^y ady,
\label{eq:a}
\end{equation}

with the following conditions:
\begin{equation} \label{eq:b}
\begin{cases}
a(0, y) = a_0, \quad a_0 \in C^6(\mathbb R_+) \cap W^{4,1}(\mathbb R_+), \\[6pt]
a(t, 0) = 0, \quad a_y(t, 0) = 0, \\[6pt]
\lim_{y\to\infty} a_{y^{(m)}}(t,y) = 0, \quad  0\le m \le 3. \\[6pt]
\partial_y^{m} a_0(0) = 0, \quad  m = 0, \,1,\, 4,\, 5.
\end{cases} 
\end{equation}
Then, there exists a $T > 0$ such that a unique solution $a$ exists, satisfying
$a \in C([0,T]; C^6(\mathbb R_+)) \cap C([0,T]; W^{4,1}(\mathbb R_+))$.

\end{prop}

\begin{proof}
Use the Duhamel formula:
\begin{equation}
\label{eq:duhamel}
    a(t) = \mathcal{T}[a](t) := S(t)a_0 + \int_0^t S(t-\tau) N(a(\tau)) \, d\tau,
\end{equation}
where $N(a) = a^2 - a_y \int_0^y a(z) \, dz$. We define the solution space
\begin{equation}
X_T = \{ a \in C([0,T]; C^1(\mathbb{R}_+) \cap W^{4,1}(\mathbb{R}_+)) \mid  a(t,0) = \partial_y^{} a(t,0) = 0 \} 
\end{equation}
with the norm 
\begin{equation}
\|a\|_{X_T} = \sup_{0 \le t \le T} (\|a\|_{L^{\infty}} + \|a_y\|_{L^{\infty}} +\|a\|_{W^{4,1}}),\,\, \|f\|_{W^{4,1}} = \sum_{k=0}^4 \|\partial_y^k f\|_{L^1}.
\end{equation}

\textbf{Step 1: Stability of the ball.}
Let $B_R = \{ a \in X_T : \|a\|_{X_T} \le R \}$. 

By Corollary~\ref{cor:idff}, the linear part satisfies $\|S(t)a_0\|_{X_T} \le C\|a_0\|_{X_T}$.

For the nonlinear part, in $L^\infty$ bounds we have:

\begin{equation}\label{eq:nonlinear_Linf}
\begin{aligned}
\|S(t-\tau)N(a)\|_{L^\infty}
&\le C'\bigl(\|a\|_{L^\infty}^2
            + \|a_y\|_{L^\infty}\|a\|_{L^1}\bigr)
 \le C\|a\|_{X^T}^2,\\[6pt]
\|\partial_y S(t-\tau)N(a)\|_{L^\infty}
&\le C'(t-\tau)^{-1/4}\|a\|_{X^T}^2 .
\end{aligned}
\end{equation}

 In $L^1$ bounds we have:
\begin{equation}\label{eq:nonlinear_L1}
\begin{aligned}
\|S(t-\tau)N(a)\|_{L^1}
&\le C\bigl(\|a\|_{L^\infty}\|a\|_{L^1}
            + \|a_y\|_{L^1}\|a\|_{L^1}\bigr)
 \le C\|a\|_{X^T}^2,\\[6pt]
\|\partial_y S(t-\tau)N(a)\|_{L^1}
&\le C(t-\tau)^{-1/4}\|a\|_{X^T}^2,\\[6pt]
\|\partial_y^2 S(t-\tau)N(a)\|_{L^1}
&\le C(t-\tau)^{-1/2}\|a\|_{X^T}^2,\\[6pt]
\|\partial_y^3 S(t-\tau)N(a)\|_{L^1}
&\le C(t-\tau)^{-3/4}\|a\|_{X^T}^2 .
\end{aligned}
\end{equation}

By equation \eqref{eq:smoothing3},

\begin{equation}
    \partial_y^4S(t-\tau)N(a)  = -\int_{0}^{\infty} \partial_y^{3} K_a(t,y,z)\,\partial_zN(a(z))\,dz.
\end{equation}

By Lemma~\ref{lem:3}, $\int_{0}^{\infty} |\partial_y^{3} K_a(t,y,z)dy| \le Ct^{-3/4}$,
\begin{equation}
    \|\partial_y^4S(t-\tau)N(a)\|_{L^1}  \le C(t-\tau)^{-3/4}\|\partial_yN(a(y))\|_{L^1}  \le C(t-\tau)^{-3/4}\|a\|_{X^T}^2.
\end{equation}

 Thus,
\begin{equation}
    \|\mathcal{T}[a]\|_{X_T} \le C_0 \|a_0\|_{X_T} + C_1 (T + T^{3/4} +T^{1/2} + T^{1/4})  R^2.
\end{equation}
Choosing $R = 2C_0 \|a_0\|_{X_T}$ and $T > 0$ small enough, we obtain $\mathcal{T} : B_R \to B_R$.

\medskip
\textbf{Step 2: Contraction property.} For any $a, b \in B_R$, the difference of the nonlinear terms is given by:
\begin{equation}
    N(a) - N(b) = (a+b)(a-b) - (a-b)_y \int_0^y a \, dz - b_y \int_0^y (a-b) \, dz.
\end{equation}

Applying the $X_T$ norm and using the same Sobolev estimates, we obtain:
\begin{equation}\label{eq:nonlinear_diff}
\begin{aligned}
\|N(a)-N(b)\|_{L^\infty}
&\le 2C'R\,\|a-b\|_{X_T},\\
\|N(a)-N(b)\|_{L^1}
&\le 2C'R\,\|a-b\|_{X_T},\\
\|\partial_y\bigl(N(a)-N(b)\bigr)\|_{L^1}
&\le 2C'R\,\|a-b\|_{X_T}.
\end{aligned}
\end{equation}

Therefore,
\begin{equation}
    \|\mathcal{T}[a] - \mathcal{T}[b]\|_{X_T} \le C  (T + T^{3/4} +T^{1/2} + T^{1/4})  R \|a-b\|_{X_T}.
\end{equation}
For sufficiently small $T$, $\mathcal{T}$ become a contraction mapping. By the Banach fixed point theorem, there exists a unique solution $a \in X_T$. 

\medskip
\textbf{Step 3: Smoothing.}
Starting from the solution $a \in C([0, T]; C^1(\mathbb{R}_+))$, we bootstrap the regularity up to $C^6$ through iterative integral estimates. We define the modified operator $S_b$ such that:
\begin{equation}
 ( S_b(t)f)(x) = \int_{0}^{\infty} {K_b}(t,x,y) f(y) \, dy.
\end{equation}
The following estimates ensure the uniform $L^\infty$-boundedness of each derivative, 

\begin{itemize}
    \item \textbf{Estimates for $\partial_y^2 a$ and $\partial_y^3 a$:}
    Since $N(a) \in L^\infty$, we apply the derivatives directly to the kernel:
\begin{equation}\label{eq:duhamel_Linf_high}
\begin{aligned}
\|\partial_y^{2} a(t)\|_{L^\infty}
&\le \|\partial_y^{2} S(t)a_0\|_{L^\infty}
   + \int_{0}^{t} \|\partial_y^{2} S(t-\tau) N(a(\tau))\|_{L^\infty}\, d\tau,\\
\|\partial_y^{3} a(t)\|_{L^\infty}
&\le \|\partial_y^{3} S(t)a_0\|_{L^\infty}
   + \int_{0}^{t} \|\partial_y^{3} S(t-\tau) N(a(\tau))\|_{L^\infty}\, d\tau .
\end{aligned}
\end{equation}

    \item \textbf{Estimates for $\partial_y^4 a$ and $\partial_y^5 a$:}
    Since derivatives of $a$ up to order three are bounded, it follows that
$\partial_y^2 N(a) \in L^\infty$ And we utilize $S_b$ to balance the derivative load (The fulfillment of the compatibility conditions $N(a)(0) = \partial_y N(a)(0) = 0$ justifies the transfer of derivatives to the kernel, as established in Corollary~\ref{cor:idff}):
\begin{equation}\label{eq:duhamel_Linf_veryhigh}
\begin{aligned}
\|\partial_y^{4} a(t)\|_{L^\infty}
&\le \|\partial_y^{4} S(t)a_0\|_{L^\infty}
 + \int_{0}^{t} \|\partial_y^{2} S_b(t-\tau)\,\partial_y^{2} N(a(\tau))\|_{L^\infty}\, d\tau,\\
\|\partial_y^{5} a(t)\|_{L^\infty}
&\le \|\partial_y^{5} S(t)a_0\|_{L^\infty}
 + \int_{0}^{t} \|\partial_y^{3} S_b(t-\tau)\,\partial_y^{2} N(a(\tau))\|_{L^\infty}\, d\tau .
\end{aligned}
\end{equation}

    \item \textbf{Estimate for $\partial_y^6 a$:}
    Finally, the boundedness of the derivatives of $a$ up to order five
implies that $\partial_y^4 N(a) \in L^\infty$.
Hence, we obtain:
    \begin{equation}
    \|\partial_y^6 a(t)\|_{L^\infty} \le \|\partial_y^6 S(t)a_0\|_{L^\infty} + \int_0^t \|\partial_y^2 S(t-\tau) \partial_y^4 N(a(\tau))\|_{L^\infty} \, d\tau.
    \end{equation}
\end{itemize}

By Corollary~\ref{cor:idff}, The terms $\partial_y^k S(t)a_0$ for $k \in \{2, \dots, 6\}$ are uniformly bounded by the higher-order norms of the initial data, $C\|a_0\|_{C^k}$. By the smoothing property (Corollary~\ref{cor:Estimation10}), the kernel singularity is bounded by:
\begin{equation}
\|\partial_y^m S(t-\tau)\|_{L^1 \to L^\infty} \le \frac{C}{(t-\tau)^{m/4}}
\end{equation}
In our construction, the maximum derivative applied to the kernel is $m=3$, yielding a singularity of $(t-\tau)^{-3/4}$. Since $3/4 < 1$, the Duhamel integrals converge, and the solution $a(t,y)$ remains uniformly bounded in $C^6$ on $[0, T]$. Furthermore, since both the linear term $S(t)a_0$ and the nonlinear term $N(a)$ remain in $L^1(\mathbb{R}_+)$, it follows from the decay property of the kernel ((iii) part of the Theorem~\ref{prop:halfline_kernel}) and the Duhamel formula \eqref{eq:duhamel} that the solution $a(t, y)$ and its spatial derivatives up to order 3 satisfy the decay condition at infinity.
\end{proof}

\begin{remark}
Since the nonlinear term $N(a)$ satisfies the compatibility conditions
only up to $N(a)=0$ and $\partial_y N(a)=0$ at the boundary,
higher-order compatibility conditions are not available.
Consequently, the integration-by-parts argument closes only up to the
seventh derivative, provided that the additional initial regularity
$\partial_y^7 a_0 \in L^\infty$ is assumed, while higher-order regularity
cannot be guaranteed within the present framework.
\end{remark}

\subsection{Blow-up of the Solution}
In this paper, blow-up means that the solution cannot remain in the solution
class $X_T$ for all times. We adopt the approach developed in \cite{EEngquist1997, Asano1988ZeroViscosity} to establish the blow-up of solutions. 

\begin{prop} \label{prop:blowup}
Define

\begin{equation}\label{eq:EF}
\begin{aligned}
& F(a) = \int_0^\infty a^2 dy,\,\,\,E(a) = \int_0^\infty (\frac{1}{2}a_{yy}^2\,\,-\ \frac1{12}a^3)dy.
\end{aligned}
\end{equation}

Assume that the initial datum $a_0$ satisfies the conditions of \eqref{eq:b} and has $E(a_0) < 0$. Then there exists a finite time $T^*$ such that 
\begin{equation}
\lim_{t\rightarrow T^* }\|a(\cdot,t)\|_{X_T(\mathbb R^+)}=+\infty.
\end{equation}


\begin{proof}
Assuming that $\|a(\cdot,t)\|_{X_T(\mathbb{R}^+)}$ remains finite,
we derive an inequality between $F$ and $E$.
First, differentiating $F$ with respect to time yields:
\begin{equation}\label{eq:difff}
\frac{dF}{dt}
= -2 \int_0^\infty a_{yy}^2 \, dy
+ 3 \int_0^\infty a^3 \, dy .
\end{equation}

Second, when differentiating $E$ with respect to time, we split $E$ into two terms. First term:
\begin{equation}
\begin{aligned}
&\frac{d}{dt} \int_0^\infty \frac{1}{2} a_{yy}^2 \, dy \\[6pt]
&=
\int_0^\infty
\left(
- a_{yyyy}^2
+ \frac{1}{2} a^2 a_y
\right) 
\, dy\,+ \,a_{yyy}(0,t) a_{yyyy}(0,t) - a_{yy}(0,t)a_{y^{(5)}}(0,t) . \\[6pt]
\end{aligned}
\end{equation}

By the equation \eqref{eq:a}, and differentiating the boundary condition \eqref{eq:b} with respect to the time, boundary term becomes:

\begin{equation}
\begin{aligned}
a_{yyy}(0,t) a_{yyyy}(0,t)-a_{yy}(0,t)a_{y^{(5)}}(0,t)&= - \,a_{yyy}(0,t) a_{t}(0,t) + a_{yy}(0,t)a_{yt}(0,t),  \\[6pt]
&= 0.
\end{aligned}
\end{equation}

Thus, 

\begin{equation}
\begin{aligned}
\frac{d}{dt} \int_0^\infty  \frac{1}{2}a_{yy}^2 \, dy 
=
\int_0^\infty
\left(
- a_{yyyy}^2
+ \frac{1}{2} a_{yy}^2 a
\right) 
\, dy .
\label{eq:lets}
\end{aligned}
\end{equation}

Second term of $E$:

\begin{equation}
\frac{d}{dt} \int_0^\infty \frac{1}{12} a^3 \, dy
=
\int_0^\infty
\left(
- \frac{1}{2}a_{yy}^2 a
+ \frac{1}{3} a^4
\right)
\, dy.
\label{eq:gitit}
\end{equation}

 Combining \eqref{eq:lets} and \eqref{eq:gitit}, with integration by parts, we get:

\begin{equation}
\frac{dE}{dt}
=
-\int_0^\infty
\left(
 (a_{yyyy} - \frac14a^2)^2 
+ \frac{13}{48} a^4
\right)
\, dy .
\label{eq:gitit2}
\end{equation}

Next, we compute the time derivative of $G(a) = -\frac{E(a)}{F(a)^{\beta}}.$
First, we have
\begin{equation}
\begin{aligned}\label{eq:...}
- F\frac{dE}{dt}
&= \int_0^\infty a^2 \, dy
   \int_0^\infty \left( \frac{1}{4}a^2 - a_{yyyy} \right)^2 dy 
   + \frac{13}{48} 
     \int_0^\infty a^2 \, dy
     \int_0^\infty a^4 \, dy, \\[6pt]
&= \int_0^\infty a^2 \, dy
   \int_0^\infty
   \left(
     a_t - \frac{3}{4} a^2
     + a_y \int_0^y a \, dz
   \right)^2dy \\[6pt]
   &+ \frac{13}{48}
     \int_0^\infty a^2 \, dy
     \int_0^\infty a^4 \, dy, \\[6pt]
&\ge
\left(
  \int_0^\infty
  a \left(
    a_t - \frac{3}{4} a^2
    + a_y \int_0^y a \, dz
  \right) dy
\right)^2
+ \frac{13}{48}
  \left( \int_0^\infty a^3 \, dy \right)^2, \\[6pt]
&=
\frac{1}{4}
\left(
   \frac{dF}{dt}
  - \frac{5}{2} \int_0^\infty a^3 \, dy
\right)^2
+ \frac{13}{48}
  \left( \int_0^\infty a^3 \, dy \right)^2.
\end{aligned}
\end{equation}

And, 

\begin{equation}
\begin{aligned} \label{eq:wow}
\frac{dF}{dt} - \frac52 \int_0^\infty a^3 \, dy
&= \int_0^\infty a_{yy}^2 \, dy - 6E
\;\ge\; -6E,
\\[6pt]
\int_0^\infty a^3 \, dy 
&= 6\int_0^\infty a_{yy}^2 \, dy - 12E
\;\ge\; -12E.
\end{aligned}
\end{equation}

Substituting \eqref{eq:wow} into \eqref{eq:...}, we have:
\begin{equation}
\begin{aligned}
- F \frac{dE}{dt}
&\ge
- \frac{3}2 E
\left(
  \frac{dF}{dt}
  - \frac94 \int_0^\infty a^3 \, dy
\right)
- \frac{13}{4} E \int_0^\infty a^3 \, dy,
\\[6pt]
&=
- E
\left(
  \frac{3}2 \frac{dF}{dt}
  - \frac18 \int_0^\infty a^3 \, dy
\right).
\end{aligned}
\end{equation}

By \eqref{eq:wow}, equation \eqref{eq:difff} becomes:
\begin{equation} \frac{dF}{dt} = -4E + \frac{8}{3}\int_0^\infty a^3  dy \ge \frac{8}{3}\int_0^\infty a^3  dy. \end{equation}

If we choose  $\beta \in (1,\,93/64)$ and use the fact that $E<0$ and $\int_0^\infty a^3  dy>0$, we get:
\begin{equation}
\begin{aligned}
- F \frac{dE}{dt} + \beta E \frac{dF}{dt}
&\ge
- E \left\{
\left( \frac{3}2 - \beta \right)\frac{dF}{dt}
- \frac18 \int_0^\infty a^3 \, dy
\right\},
\\[6pt]
&\ge
- E \left( \frac{31}{8} - \frac{8}{3}\beta \right)
\int_0^\infty a^3 \, dy
\;\ge\; 0 .
\end{aligned}
\end{equation}

And we have $\frac{dG}{dt}\ge 0$. Therefore,
\begin{equation}\label{eq:2.9}
\frac{dF}{dt}\ge -6E \ge 6\,G(a_0)\,F^\beta .
\end{equation}

Hence there exists a finite time $T^{*}$ such that $\lim_{t\to T^{*}} F(a)=+\infty$, This completes the proof of Theorem~\ref{prop:blowup}.









\end{proof}
\end{prop}

\subsection{Numerical Example}
To illustrate the blow-up phenomenon established in Theorem~\ref{prop:blowup}, we consider a smooth, compactly supported initial datum $a_0 \in C^\infty(\mathbb{R}_+)$ defined by
\begin{equation} \label{eq:example}
a_0(y) = 
\begin{cases} 
10 \exp\left( -\frac{1}{1 - \left(\frac{y-20}{10}\right)^2} \right), & \text{if } |y-20| < 10, \\
0, & \text{if } |y-20| \ge 10.
\end{cases}
\end{equation}
This bump function is centered at $y_0=20$ with a support in $(10, 30)$. By computing the $E(a_0)$ as defined in \eqref{eq:EF}, $E(a_0) \approx -34.9 < 0$. According to Theorem 4, this ensures the existence of a finite-time blow-up for the corresponding solution.

\begin{figure}[H]
    \centering
    \includegraphics[width=0.75\linewidth]{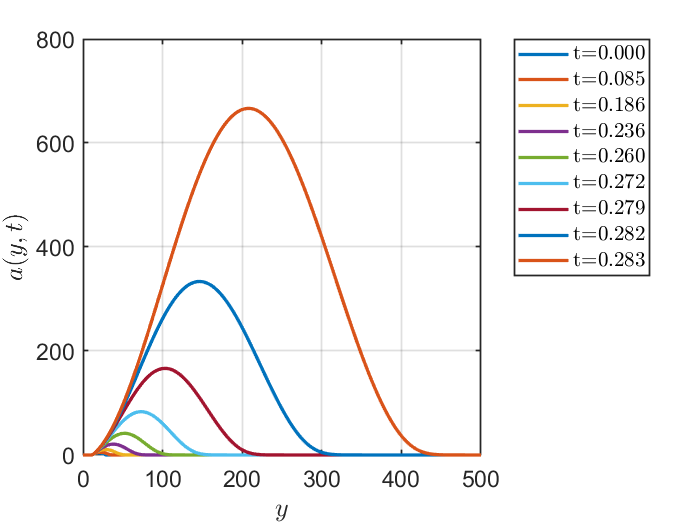}
    \caption{The evolution of the solution profile for the initial data \eqref{eq:example}. 
    Profiles are sampled at time instances $\{t_n\}$ such that $\|a(\cdot, t_n)\|_{L^\infty} = 2^n \|a(\cdot, 0)\|_{L^\infty}$.}
    \label{fig:blowup_evolution}
\end{figure}

\bibliographystyle{amsplain}
\bibliography{references}

\providecommand{\bysame}{\leavevmode\hbox to3em{\hrulefill}\thinspace}
\providecommand{\MR}{\relax\ifhmode\unskip\space\fi MR }
\providecommand{\MRhref}[2]{%
  \href{http://www.ams.org/mathscinet-getitem?mr=#1}{#2}
}
\providecommand{\href}[2]{#2}
\begin{thebibliography}{1}

\bibitem{Asano1988ZeroViscosity}
Kazuo Asano, \emph{Zero-viscosity limit of the incompressible navier--stokes equation, ii}, Mathematical Analysis of Fluid and Plasma Dynamics I, S\^urikaisekikenky\^usho K\^oky\^uroku, vol. 656, Research Institute for Mathematical Sciences, Kyoto, 1988, pp.~105--128.

\bibitem{Chatziafratis2025HigherOrder}
A.~Chatziafratis, A.~Miranville, G.~Karali, A.~S. Fokas, and E.~C. Aifantis, \emph{Higher-order diffusion and {C}ahn--{H}illiard-type models revisited on the half-line}, Mathematical Models and Methods in Applied Sciences \textbf{35} (2025), no.~5, 1133--1197.

\bibitem{EEngquist1997}
Weinan E and Bj{\"o}rn Engquist, \emph{Blowup of solutions of the unsteady {P}randtl equation}, Communications on Pure and Applied Mathematics \textbf{50} (1997), no.~12, 1287--1293.

\bibitem{GerardVaretDormy2010}
David G{\'e}rard-Varet and Emmanuel Dormy, \emph{On the ill-posedness of the prandtl equation}, Journal of the American Mathematical Society \textbf{23} (2010), no.~2, 591--609.

\bibitem{KukavicaVicolWang2017}
Igor Kukavica, Vlad Vicol, and Fei Wang, \emph{The van dommelen and shen singularity in the prandtl equations}, Advances in Mathematics \textbf{307} (2017), 288--311.

\bibitem{OzsariYolcu2019CPAA}
T{\"u}rker \"{O}zsar{\i} and Nermin Yolcu, \emph{The initial-boundary value problem for the biharmonic schr{\"o}dinger equation on the half-line}, Communications on Pure and Applied Analysis \textbf{18} (2019), no.~6, 3285--3316.

\bibitem{SammartinoCaflisch1998a}
Marco Sammartino and Russel~E. Caflisch, \emph{Zero viscosity limit for analytic solutions of the navier--stokes equation on a half-space. i. existence for euler and prandtl equations}, Communications in Mathematical Physics \textbf{192} (1998), no.~2, 433--461.

\end{thebibliography}
\end{document}